\documentclass{bcp}
\usepackage{amsmath,amsthm}
\usepackage{amssymb}
\usepackage{xcolor}
\usepackage[all]{xy}

\usepackage{graphicx}

\newcommand{\C}{\mathbb{C}}

\newcommand{\Q}{\mathbb{Q}}
\newcommand{\Z}{\mathbb{Z}}
\renewcommand{\P}{\mathbb{P}}

\newcommand{\fa}{\mathfrak{a}}
\newcommand{\fg}{\mathfrak{g}}
\newcommand{\fgl}{\mathfrak{gl}}
\newcommand{\fo}{\mathfrak{o}}
\newcommand{\fsp}{\mathfrak{sp}}
\newcommand{\fz}{\mathfrak{z}}

\newcommand{\cF}{\mathcal{F}}
\newcommand{\cW}{\mathcal{W}}

\newcommand{\bD}{\mathbf{D}}

\newcommand{\even}{\mathrm{even}\,}
\newcommand{\Hom}{\mathrm{Hom}}
\newcommand{\odd}{\mathrm{odd}\,}
\newcommand{\prim}{\mathrm{Prim}\,}
\newcommand{\tr}{\mathrm{tr}}
\newcommand{\Tr}{\mathrm{Tr}}
\newcommand{\rank}{\mathrm{rank}}
\newcommand{\Rank}{\mathrm{Rank}}
\newcommand{\vac}{\vert\,\mathrm{vac}\rangle}

\newtheorem{theorem}{Theorem}[section]
\newtheorem{proposition}[theorem]{Proposition}
\newtheorem{lemma}[theorem]{Lemma}
\newtheorem{corollary}[theorem]{Corollary}

\theoremstyle{definition}

\newtheorem{example}[theorem]{Example}
\newtheorem{remark}[theorem]{Remark}

\begin{document}

\mathclass{Primary 17B65, 16S35; Secondary 16E40.}

\abbrevauthors{A. Fialowski and K. Iohara}
\abbrevtitle{Generalized Jacobi matrices}

\title{On Lie algebras of generalized Jacobi matrices}

\author{Alice Fialowski}
\address{University of P\'{e}cs and E\"{o}tv\"{o}s Lor\'{a}nd University\\
Budapest, Hungary\\
E-mail: fialowsk@ttk.pte.hu, fialowsk@cs.elte.hu}

\author{Kenji Iohara}
\address{Univ Lyon, Universit\'{e} Claude Bernard Lyon 1\\
CNRS UMR 5208, Institut Camille Jordan, \\
F-69622 Villeurbanne, France\\
E-mail: iohara@math.univ-lyon1.fr}

\maketitlebcp

\begin{abstract}
{ In these lecture notes, we consider infinite dimensional Lie algebras of generalized Jacobi matrices $\fg J(k)$ and $\fgl_\infty(k)$, which are important in soliton theory, and their orthogonal and symplectic subalgebras. In particular,  we construct the homology ring of the Lie algebra $\fg J(k)$ and of the orthogonal and symplectic subalgebras.}
\end{abstract}

\section{Introduction}
\medskip

In this expository article, we consider a special type of general linear Lie algebras
of infinite rank over a field $k$ of characteristic zero, and compute their homology with trivial coefficients. 

Our interest in this topic came from an old short note (2 pages long) of Boris Feigin and Boris Tsygan (1983, \cite{FT}). They stated some results on the homology with trivial coefficients of the Lie algebra of generalized Jacobi matrices over a field of characteristic $0$. 
It was a real effort to piece out the precise statements in that densely encoded note, not to speak about the few line proofs. 
In the process of understanding the statements and proofs, we got involved in the topic, found different generalizations and figured out correct approaches. We should also mention that their old results generated much interest during these 36 years. In the meantime, two big problems have beed solved, which helped us to understand the statements and work out the right proofs. One is the Loday-Quillen-Tsygan theorem (see (3.1)) and the other is the generalization of the Hochschild-Serre type spectral sequence by Stefan (see (3.4)), without which the statements of the old note could not be justified.

Consider the Lie algebra $\fg J(k)$ of generalized Jacobi matrices, namely, infinite size matrices $M=(m_{i,j})$ indexed over $\Z$ such that $m_{i,j}=0$ if $\vert i-j\vert >N$ for some $N$ depending on the matrix $M$. The original Jacobi operator, also known as Jacobi matrix, is a symmetric linear operator acting on sequences which is given by an infinite tridiagonal matrix  (a band matrix that has nonzero elements only on the main diagonal, the first diagonal below this, and the first diagonal above the main diagonal.) It is commonly used to specify systems of orthonormal polynomials over a finite, positive Borel measure. This operator is named after Carl Gustav Jacob Jacobi, who introduced in 1848 tridiagonal matrices and proved the following Theorem: every symmetric matrix over a principal ideal domain is congruent to a tridiagonal matrix (see \cite{K} and e.g. \cite{Sz}).  
Since then, Jacobi matrices play an important role in different branches of mathematics, like in topology (Bott's periodicity theorem on homotopy groups), stable homotopy theory, algebraic geometry and $C^*$-algebras (see \cite{K1}). They are also used to show some interesting properties in K-theory. For instance, Karoubi used them to prove the conjecture of Atiyah-Singer about the classifying space, see \cite{K2}.

The Lie algebra $\fgl_\infty(k)$ of finitely supported infinite size matrices indexed over $\Z$ can be naturally viewed as a subalgebra of $\fg J(k)$. 

The Lie algebra $\fg J(k)$ has typical infinite dimensional nature. For example, the two matrices
\[ P:=\sum_{i \in \Z} iE_{i-1,i}, \qquad Q:=\sum_{i \in \Z} E_{i+1,i}, \]
where $E_{i,j}$ is the matrix unit with $1$ on the $(i,j)$-entry, belong to $\fg J(k)$ but not to $\fgl_\infty(k)$.  They satisfy
\begin{enumerate}
\item ${}^tQ\cdot Q=I$, 
\item $PQ-QP=I$,
\end{enumerate}
where $I$ denotes the identity matrix 
$\sum_{i \in \Z} E_{i,i} \in \fg J(k)$. This matrix $I$ does not even belong to $\fgl_\infty(k)$ !  
(Off course, the matrices $P$ and $Q$ are matrix representations of $\dfrac{d}{dt}$ and $t \cdot $ on the vector space $\C[t^{\pm 1}]$ with respect to the basis $\{t^{i}\}_{i \in \Z}$.)

Such algebras show up in many areas of mathematics and physics. For instance, they are used to describe the solitons of the Kadomtsev-Petviashvili (KP) type hierarchies \cite{Sa} where such integrable systems are interpreted as a dynamical system on the so-called Sato Grassmannian. On the other hand, their basic algebraic properties and invariants are not well understood. In our work we present results on their homology with trivial coefficients, For this, we need to use several different (co)homology concepts. Some results were obtained by Feigin and Tsygan \cite{FT}  in 1983, but in  their short note the statements and proofs are not precise. In our work, we were able to get straightforward statements and proofs and we also could generalize the results to the coefficients over an associative unital $k$-algebra. 

The structure of the paper is as follows. In Section 2 we introduce some important classes of Lie algebras of general Jacobi matrices and recall their universal central extension. We also give some examples of its subalgebras. Section 3 is devoted to their (co)homology. First we recall the main definitions: Lie algebra homology, Hochschild homology, cyclic homology and (skew)dihedral homology. Then we compute homology with trivial coefficients for the introduced Lie algebras. In this section, we also  present precise proofs of some results in \cite{FT} and give possible generalizations. In Section 4 we introduce two important subalgebras, the orthogonal and symplectic subalgebras. To compute their (co)homology, we need to introduce additional computational methods to the previous one.  In Section 5 we discuss a more general case, where instead of the field $k$ we have an associative unital $k$-algebra, and generalize our (co)homology results for such algebras. Finally, we introduce a rank functional on a Lie subalgebra $\fg J_\infty(k) \subseteq \fg J(k)$ and describe its image. After describing its cohomology ring,  we also raise some open questions.

\section{Lie algebras of generalized Jacobi matrices}\label{sect_2}

\subsection{Lie algebras $\fgl_\infty(k)$ and $\fg J(k)$}

 The first such Lie algebra one may have in mind is the one defined as an inductive limit: let $I$ be a countable set and $I_1 \subsetneq I_2 \subsetneq \cdots \subsetneq I_n \subsetneq \cdots$, $I=\bigcup_n I_n$ an increasing series where each $I_n$ is a finite subset.  The Lie algebra $\fgl_I(k)$ is defined by the inductive limit of $\fgl_{I_m}(k) \hookrightarrow \fgl_{I_n}(k)$ for $m<n$, namely, each element of $\fgl_I(k)$ is a finitely supported matrix of infinite size indexed over $I$, i.e., $X=(x_{i,j})_{i,j \in I}$ such that $\sharp \{(i,j) \in I^2\, \vert \, x_{i,j}\neq 0\, \}<\infty$. In particular, for $I=\Z$, we may denote by $\fgl_\infty(k)$. 
 
 Another one we may also encounter is the Lie algebra of generalized Jacobi matrices defined as follows. A generalized Jacobi matrix is a matrix $M=(m_{i,j})$ indexed over $\Z$ such that there exists a positive integer $N_M$ satisfying
 \[  m_{i,j}=0 \qquad \forall\; i,j \quad \text{such that} \quad \vert i-j\vert>N_M. \]
The set of such matrices has a structure of associative algebra over $k$ denoted by $J(k)$. We shall denote it by $\fg J(k)$ whenever we regard it as Lie algebra. An original Jacobi matrix is a finite size matrix $M=(m_{i,j})$ such that $m_{i,j}=0$ for any $i,j$ with $\vert i-j \vert >1$. 
The Lie algebra $\fgl_\infty(k)$ can be naturally viewed as a subalgebra of $\fg J(k)$. 

\subsection{Universal central extension of the Lie algebra $\fg J(k)$}\label{sect_UCE}
In the course of studying soliton theory, a non-trivial central extension of the Lie algebra $\fg J(k)$ was discovered (see, e.g., \cite{JM} and \cite{DJM}) and it can be described as follows. 

Let $J=\sum_{i\geq 0}E_{i,i} \in \fg J(k)$ be a matrix and let $\Phi: J(k) \rightarrow J(k)$ be the $k$-linear map defined by $X \mapsto JXJ$. It can be checked that, for any $X, Y \in J(k)$, the element $[\Phi(X), \Phi(Y)]-\Phi([X,Y])$ is an element of $\fgl_\infty(k)$, i.e., only finitely many matrix entires can be non-zero. Hence, one can define the $k$-bilinear map
$\Psi: J(k) \times J(k) \rightarrow k$ by
\[ \Psi(X,Y)=\tr([\Phi(X), \Phi(Y)]-\Phi([X,Y])), \]
where $\tr: \fgl_\infty( k) \rightarrow k \, ; \, X=(x_{i,j}) \mapsto \sum_{i \in \Z} x_{i,i}$ is the trace of finitely supported matrices. It turned out that this $\Psi$ is a $2$-cocycle, called \emph{Japanese cocycle}, i.e., 
\begin{enumerate}
\item $\Psi(Y,X)=-\Psi(X,Y)$, 
\item $\Psi([X,Y],Z)+\Psi([Y,Z], X)+\Psi([Z,X],Y)=0$,
\end{enumerate} 
for any $X,Y, Z \in J(k)$. Let $\widetilde{\fg J}(k):=\fg J(k) \oplus k\cdot 1$ be the Lie algebra whose Lie bracket $[ \cdot, \cdot ]'$ is given by
\[ [X,1]'=0, \qquad [X, Y]'=[X,Y]+\Psi(X,Y)1 \qquad X, Y \in \fg J(k). \]
As the Lie algebra $\fg J(k)$ is perfect, i.e., $[\fg J(k), \fg J(k)]=\fg J(k)$, the Lie algebra $\fg J(k)$ admits the universal central extension. 
It was B. L. Feigin and B. L. Tsygan \cite{FT} in 1983 who proved that this central extension $\alpha: \widetilde{\fg J}(k) \rightarrow \fg J(k)$ is \emph{universal,} namely, $\widetilde{\fg J}(k)$ is perfect and for any central extension $\beta: \fa \rightarrow \fg J(k)$, there exists a morphism of Lie algebras $\gamma: \widetilde{\fg J}(k) \rightarrow \fa$ such that the next diagram commutes:
\[ \UseTips
\xymatrix @=1.2pc @R=10pt @C=15pt
{
\widetilde{\fg J}(k) \ar[rr]^{\alpha} \ar[rdd]_{\gamma} && \fg J (k) \\
&& \\
& \fa \ar[ruu]_{\beta} & \\
}\]
\begin{remark} The kernel $\fz$ of the universal central extension of $\fg J(k)$, i.e., the kernel of the canonical projection $\widetilde{\fg J}(k) \twoheadrightarrow \fg J(k)$ can be given by the $2$nd homology $H_2(\fg J(k))$ and the $2$-cocycle $\Psi$ is an element of the $2$nd cohomology $H^2(\fg J(k))$. 
\end{remark}
\begin{remark} Let $\widetilde{G_J}$ be the (pro-)algebraic group of $\widetilde{\fg J}(k)$. The Lie algebra $\widetilde{\fg J}(k)$ acts on a fermionic Fock space $\cF$. The $\widetilde{G_J}$-orbit of its vacuum state $\vac$ in the projective space $\P\cF$  is isomorphic to the Sato Grassmannian. The defining equations of this orbit in terms of Pl\"{u}cker coordinates is nothing but the Hirota bilinear equations of Kadomtsev-Petviashvili (KP) hierarchy. For details, see, e.g., \cite{JM} and \cite{DJM}. 
\end{remark}

\subsection{Some subalgebras of the universal central extension $\widetilde{\fg J}(k)$}
The extended Lie algebra contains several interesting infinite dimensional Lie algebras as subalgebras. We shall introduce some of them. \\

1. For an integer $n>1$, let $\widetilde{\fg J_n}(k)$ be the subalgebra of $\widetilde{\fg J}(k)$ generated by the matrices $M=(m_{i,j})_{i,j \in \Z}$ satisfying $m_{i+n,j+n}=m_{i,j}$ for any $i, j \in \Z$. This subalgebra is isomorphic to the central extension of $\fgl_n(k[t,t^{-1}])$, i.e., the affine Lie algebra 
\[ \widehat{\fgl_n}(k)=\fgl_n(k[t,t^{-1}])\oplus kc \]
whose commutation relation is given by
\[ [e_{i,j}\otimes t^{a}, e_{k,l}\otimes t^{b}]=[e_{i,j}, e_{k,l}]\otimes t^{a+b}+a\delta_{a+b,0}\delta_{j,k}\delta_{i,l}c, \qquad [\widehat{\fgl_n}(k), c]=0, \]
where $e_{i,j} \in \fgl_n$ is the matrix unit whose $(i,j)$-entry is $1$. An isomorphism between $\widetilde{\fgl_n}(k)$ and $\widetilde{\fg J_n}(k)$ is given by
$e_{i,j}\otimes t^{a} \mapsto \sum_{r \in \Z}e_{i+rn, j+(r+a)n}$.  \\

2. Another important example is the one-dimensional central extension of the Lie algebra $k[t,t^{-1}]\left[\dfrac{d}{dt}\right]$ of algebraic differential operators over $k^\ast=k \setminus \{0\}$ that is defined as follows. Set $D=t\frac{d}{dt}$. For any polynomial $f, g \in k[D]$, it can be verified that
\[ [t^rf(D), t^sg(D)]=t^{r+s}(f(D+s)g(D)-f(D)g(D+r)). \]
Let $\Psi$ be the $2$-cocyle on $k[t,t^{-1}]\left[\dfrac{d}{dt}\right]$ defined by
\[ \Psi(t^r f(D), t^sg(D)):=\begin{cases} \sum_{-r\leq j \leq -1}f(j)g(j+r) \qquad &r=-s \geq 0, \\
\quad 0 \quad &r+s\neq 0. \end{cases}
\]
The Lie algebra $\cW_{1+\infty}$ is, by definition, the central extension of $k[t,t^{-1}]\left[\dfrac{d}{dt}\right]$ by the $2$-cocycle $\Psi$, i.e., it is the $k$-vector space 
\[ \cW_{1+\infty} =k[t,t^{-1}]\left[\dfrac{d}{dt}\right] \oplus kC, \]
equipped with the Lie bracket given by
\[ [t^r f(D), t^s g(D)]'=[t^r f(D), t^s g(D)]+\Psi(t^r f(D), t^s g(D))C, \qquad [\cW_{1+\infty}, C]=0.
\]
It can be shown that there exists morphism of Lie algebras $\cW_{1+\infty} \hookrightarrow \widetilde{\fg J}(k)$ satisfying $t^a D^r \, \mapsto \sum_{i \in \Z} i^{a}E_{i+r, i}$. We remark that the Lie subalgebra $k[t,t^{-1}]D \oplus k C$ of $\cW_{1+\infty}$ is isomorphic to the Virasoro algebra. 

The first example shows that the Lie algebra $\widetilde{\fg J}(k)$ contains, at least, affine Lie algebras of classical type, i.e.,  $A_l^{(1)} (l\geq 1), B_l^{(1)} (l\geq 3), C_l^{(1)} (l\geq 2), D_l^{(1)} (l\geq 4), A_{2l}^{(2)} (l\geq 1), A_{2l-1}^{(2)}, (l\geq 3), D_{l+1}^{(2)} (l\geq 2)$ and $D_4^{(3)}$. In addition, the second example shows that Lie algebra $\widetilde{\fg J}(k)$
contains also the Lie algebra $\cW_{1+\infty}$ that plays an important role in the KP-hierarchy. 

\section{Homology of the Lie algebra $\fg J(k)$}\label{sect_3}
\medskip

In this Section, among others we state the main result of \cite{FT} and explain the outline of the proof. 
Our may goal is the computation of the homology $H_\bullet(\fg J(k))$. 

\subsection{Several homologies}
We briefly recall the definitions of Lie algebra (co)homology, Hochschild homology, cyclic homology and (skew-)dihedral homology. \\

Let $\fg$ be a Lie algebra over a field $k$ of characteristic $0$. From now on, we shall abbreviate the coefficient $k$. The \emph{Lie algebra homology} $H_\bullet(\fg)$ is, by definition, the homology of the complex $(\bigwedge^\bullet \fg, d)$, called the \emph{Eilenberg-Chevalley complex}, where $\bigwedge^\bullet \fg$ is the exterior algebra of $\fg$ and the differential $d$ is given by
\[ d(x_1\wedge \cdots \wedge x_n)=\sum_{1\leq i<j\leq n}(-1)^{i+j+1}[x_i,x_j]\wedge x_1\wedge \cdots \hat{x_i}\wedge \cdots \wedge \hat{x_j} \wedge \cdots \wedge x_n. \]
The Lie algebra cohomology $H^\bullet(\fg)$ is by definition, the cohomology of the `dual of the complex $(\bigwedge^\bullet \fg, d)$' , i.e., $(\Hom_{k}(\bigwedge^\bullet \fg, k),-{}^td)$. 
For some basic properties of this homology, see, e.g. \cite{HS}. 
The homology $H_\bullet(\fg)$ has a commutative and cocommutative DG-Hopf algebra structure. Its coalgebra structure with counit is induced from the comultiplication
\[ \Delta: H_\bullet(\fg) \overset{\delta}{\longrightarrow} H_\bullet(\fg \oplus \fg) \overset{\sigma}{\longrightarrow} H_\bullet(\fg) \otimes H_\bullet(\fg), 
\]
where $\delta: \fg \rightarrow \fg \oplus \fg; \, x \, \mapsto (x,x)$ is the diagonal map, and $\sigma$ is the K\"{u}nneth isomorphism.
Hence, the homology $H_\bullet(\fg)$ is the graded symmetric algebra over its primitive part (cf. \cite{Q}). 
Recall that an element $x \in H_\bullet(\fg)$ is said to be \emph{primitive} if 
its cocommutative coalgebra structure with counit is induced from the comultiplication which
satisfies $\Delta(x)=x\otimes 1+1\otimes x$. 
For the Lie algebra $\fgl_\infty(R)$ over an associative unital $k$-algebra $R$, the primitive part of the homology $H_\bullet(\fgl_\infty(R))$ had been known by Loday and Quillen \cite{LQ} and independently by B. L. Tsygan \cite{T} as follows:
\begin{theorem}\label{thm_LQT}
The primitive part $\prim H_\bullet(\fgl_\infty(R))$ is isomorphic to the cyclic homology $HC_{\bullet-1}(R)$. 
\end{theorem}
The cohomology ring $H^\bullet(\fg)$ is naturally endowed with a commutative and cocommutative DG-Hopf algebra structure. Indeed, by the Poincar\'{e} duality $H^n(\fg) \cong H_n(\fg)^\ast$ (full dual), one has $H^\bullet(\fg) \cong H_\bullet(\fg)^\vee:=\bigoplus_n H_n(\fg)^\ast$ (restricted dual). 

Now we recall some definitions in homology theory of associative algebras. For detail, see, e.g., \cite{Lod}.\\

Let $R$ be an associative unital $k$-algebra and $M$ be an $R$-bimodule. For $n \in \Z_{\geq 0}$, set $C_n(R,M)=M\otimes R^{\otimes n}$. The \emph{Hochschild homology} $H_\bullet(R,M)$ is, by definition, the homology of the complex $(C_\bullet(R,M), b)$, where the the differential $b$ is defined by
\begin{align*}
&b(m\otimes r_1\otimes \cdots \otimes r_n) \\
=
&ma_1\otimes a_2\otimes \cdots \otimes a_n+\sum_{i=1}^{n-1}(-1)^{i}m \otimes a_1 \otimes \cdots \otimes a_ia_{i+1}\otimes \cdots \otimes a_n \\
&+(-1)^{n}a_nm\otimes a_1\otimes \cdots \otimes a_{n-1}.
\end{align*}
In particular, for $M=R$, for each $n \in \Z_{>0}$, there is an action of the cyclic group $\Z/(n+1)\Z$ given by
\[ x.(r_0\otimes r_1\otimes \cdots \otimes r_n)=(-1)^n(r_n\otimes r_0 \otimes \cdots \otimes r_{n-1}), \]
where $x$ is a generator of the group $\Z/(n+1)\Z$. The differential $b$ of the complex $(C_\bullet(R,R), b)$ induces a differential on the complex $C_n^\lambda(R):=R^{\otimes (n+1)}/(1-x)$.  The homology of this complex, called \emph{Connes' complex} is the so-called \emph{cyclic homology} $HC_\bullet(R)$ of $R$. For some cases, this cyclic homology can be computed with the aid of 
\emph{Connes' periodicity exact sequence}:
\begin{equation}\label{Connes} 
\cdots \rightarrow HH_n(R) \rightarrow HC_n(R) \rightarrow HC_{n-2}(R) \rightarrow HH_{n-1}(R) \rightarrow \cdots.
\end{equation}

Now, assume that $R$ is equipped with a $k$-linear anti-involution $\bar{\cdot}: R \rightarrow R$. 
One can extend the action of the group $\Z/(n+1)\Z$ on $R^{\otimes (n+1)}$ to the dihedral group $D_{n+1}=\langle x,y \vert x^{n+1}=y^2=1, yxy=x^{-1}\rangle$ by
\[ y.(r_0\otimes r_1\otimes \cdots \otimes r_n)=(-1)^{\frac{1}{2}n(n+1)}(\overline{r_0}\otimes \overline{r_n}\otimes \overline{r_{n-1}}\otimes \cdots \otimes \overline{r_1}). \]
Let $\bD_n(R)$ denote the space of coinvariants $(R^{\otimes n+1})_{D_{n+1}}$. If we modify the action of $y$ by $-y$, the resulting coinvariants will be denoted by ${}_{-1}\bD_n(R)$. The differential $b$ on $C_\bullet(R,R)$ induces the differentials on $(R^{\otimes n+1})_{D_{n+1}}$ and ${}_{-1}\bD_n(R)$, denoted by $\bar{b}$. Their homologies are called \emph{dihedral }(resp. \emph{skew-dihedral}) \emph{homology} of $R$:
\[ HD_n(R):=H_n(\bD_\bullet(R),\bar{b}), \qquad 
  (\; \text{resp}.\; {}_{-1}HD_n(R):=H_n({}_{-1}\bD_\bullet(R),\bar{b})\; ).
\]
For more informations, see, e.g., \cite{Lod}. 
\subsection{An isomorphism $\fg J(k) \cong \fgl_n(J(k))$}
Let $n>1$ be an integer. We fix a section of $\Z \twoheadrightarrow \Z/n\Z$ and denote its image by $I$. For any matrix $M=(m_{i,j})_{i,j \in \Z}$ and any $i,j \in I$, we set $M_{i,j}:=(m_{i+nk, j+nl})_{k,l \in \Z}$. Then the $k$-linear map 
\[ \Phi_I: \fg J(k) \longrightarrow \fgl_n(J(k));  \qquad M \; \longmapsto \; (M_{i,j})_{i,j \in I}, \]
is an isomorphism of Lie algebras $\fg J(k) \longrightarrow \fgl_n(J(k))$, and it induces an isomorphism of homologies $H_\bullet(\fg J(k)) \cong H_\bullet(\fgl_n(J(k)))$. Taking an inductive limit, we obtain

\begin{lemma}[Lemma 1 of \cite{FT}]\label{lemma_stable1} $H_\bullet(\fg J(k)) \cong H_\bullet(\fgl_\infty(J(k)))$.
\end{lemma}

Hence, the homology $H_\bullet(\fg J(k))$ is the commutative and cocommutative DG-Hopf algebra whose primitive part is $HC_{\bullet-1}(J(k))$ by Theorem \ref{thm_LQT}. Thus, if we can compute the Hochschild homology $HH_\bullet(J(k))$ of the associative algebra $J(k)$, Connes' periodicity exact sequence \eqref{Connes} allows us to determine the cyclic homology $HC_{\bullet}(J(k))$.

 In the rest of this section, we will explain this briefly.  

\subsection{Twisted Group Algebras}
Let $A$ be an associative unital $k$-algebra and let $G$ be a discrete subgroup of the group of $k$-automorphisms of $A$. One can twist the natural product structure on $A \otimes_k k[G]$ by $(a \otimes [g])\cdot (b \otimes [h]):=ag(b)\otimes [gh]$, where $a, b\in A$ and $g, h \in G$. The tensor product $A \otimes_k k[G]$ equipped with such a twisted product is called \emph{twisted group algebra} and will be denoted by $A\{G\}$. 

\begin{example}

\begin{enumerate}
\item Fix $n \in \Z_{>1}$. Let $A=\overbrace{k\times k\times \cdots \times k}^n$ be the $n$-copies of $k$ viewed as commutative associative algebra. The group $G=\Z/n\Z$ acts on A via cyclic permutation: $(i+n\Z).(a_1,\ldots, a_n)=(a_{i+1}, \ldots, a_{i+n})$ where the indicies should be understood modulo $n$. Then, the twisted group algebra $A\{G\}$ is isomorphic to the algebra of $n\times n$ matrices $M_n(k)$. The isomorphism $A\{G\} \rightarrow M_n(k)$ is given by
\[ (a_1,\ldots, a_n) \otimes [i+ n\Z] \quad \longmapsto \quad \sum_{k=1}^n a_ke_{k, k+i}. \]
\item $A=\prod_{i \in \Z}ke_i$, where $e_i$'s are orthogonal idempotents and $G=\Z$. 
The additive group $\Z$ acts on $A$: $1.e_i=e_{i-1}.$ One can show that the twisted group algebra $A\{G\}$ is isomorphic to the associative algebra $J(k)$ where the isomorphism $A\{G\} \rightarrow J(k)$ is given by
\[ \sum_{k \in \Z} a_ke_k \otimes [i] \quad \longmapsto \quad \sum_{k \in \Z}a_kE_{k,k+i}.
\]
\end{enumerate}
\end{example}
\subsection{Hochschild-Serre type Spectral Sequence}
Let $M$ be an $A\{G\}$-bimodule. 
\begin{theorem} \cite{St} There exists a spectral sequence
$$
 E_{p,q}^2=H_p(k[G], H_q(A, M)) \; \Longrightarrow \; H_{p+q}(A\{G\}, M)$$
 \end{theorem} 
 
\begin{remark} In \cite{FT}, B. Feigin and B. Tsygan treated the above case when $M=A\{G\}$. But their description of the $k[G]$-module structure on $H_q(A,A\{G\})$ was not correct and this was rectified by D. Stefan \cite{St} under more general setting. 
\end{remark}
In our case, $G=\Z$ and $k[\Z]$ is a principal ideal domain, hence its global dimension is at most $1$ which implies that this spectral sequence collapses at $E^2$. \\

By a standard argument, one can show that
\[ H_q(A, J(k)))\cong \prod_{i \in \Z}HH_q (k)e_i. \]
Moreover, by direct computation, it can be verified that
\[ H_p(k[\Z], H_q(A, J(k))) \cong \begin{cases}
HH_q(k) \qquad & p=1, \\
\quad 0 \quad & p\neq 1. \end{cases}
\]
This implies $HH_p(J(k))\cong HH_{p-1}(k).$
\begin{remark}
The isomorphism $H_1(k[\Z], H_p(A,J(k))\cong HH_{p+1}(J(k))$ is given by the so-called \emph{shuffle product} \cite{Q}. This fact plays an important role when we determined the homology of the Lie algebras of orthogonal and symplectic generalized Jacobi matrices in \cite{FI2}. 
\end{remark}
By definition, it follows that 
\[ HH_p(k)= \begin{cases} \quad k \quad & p=0, \\ \quad 0 \quad & p>0. \end{cases}
\]
Thus, we obtain
\begin{theorem}[cf. Theorem 3 in \cite{FT}] $HH_1(J(k))=k$ and $HH_p(J(k))=0$ for any $p\neq 1$.
\end{theorem}
By Connes' Periodicity long exact sequence \eqref{Connes}, one has
\begin{corollary} \label{cor_HC-Jacobi}
$HC_p(J(k))=k$ for odd $p$ and $HC_p(J(k))=0$ for even $p$.
\end{corollary} 

\subsection{Description of the primitive part $\prim(H_\bullet(\fg J(k)))$}
Theorem \ref{thm_LQT} and Corollary \ref{cor_HC-Jacobi} imply 
\begin{theorem}\label{thm_main1} $\prim(H_\bullet(\fg J(k)))$ is the graded $k$-vector space whose $p$th graded component is
\[ \prim(H_\bullet(\fg J(k)))_p=\begin{cases} \, k \, \quad & p \equiv 0 (2), \\ \, 0\, & \text{otherwise}. \end{cases}
\]
\end{theorem}
In particular, each graded subspace of $H_\bullet(\fg J(k))$ is of finite dimension. 
Thus, by the Poincar\'{e} duality
$H^p(\fg J(k)) \cong H_p(\fg J(k))^\ast $,  one obtains Theorem 1. a) in \cite{FT}, i.e.,  the cohomology ring $H^\bullet(\fg J(k))$ has the next description:
\begin{theorem}
There exists primitive elements $c_i \in H^{2i}(\fg J(k))$ such that $H^\bullet(\fg J(k))$ is isomorphic to the Hopf algebra $S(\bigoplus_{i \in \Z_{>0}} kc_i)$.
\end{theorem}

\section{Orthogonal and Symplectic Subalgebras of $\fg J(k)$}\label{sect_4}

In this Section, after recalling the definition of orthogonal and symplectic subalgebras of $\fg J(k)$, we present two key steps to compute their  homology of these subalgebras and state the results about their homology.

\subsection{Definitions}
Let $R$ be an associative unital $k$-algebra equipped with an $k$-linear anti-involution $\bar{\cdot}$, i.e., $\bar{\bar{s}}=s$ and $\overline{st}=\bar{t}\bar{s}$ for $s,t \in R$. We extend the \emph{transpose}, also denoted by ${}^t(\cdot )$, to $J(R)$ by ${}^t(E_{i,j}(r))=E_{j,i}(\bar{r})$, where $E_{i,j}(r)=E_{i,j}r \in J(R)$. 

For $l \in \Z$, let $\tau_l, \tau_l^s$ be the $k$-linear anti-involutions of the Lie algebra $\fg J(k)$ defined by
\begin{align*}
&\tau_l(X)=(-1)^{l}J_l ({}^tX)J_l, \qquad J_l:=\sum_{i \in \Z}(-1)^{i}E_{i,l-i}, \\
&\tau_l^s(X)=J_l^s({}^tX)J_l^s, \qquad J_l^s:=\sum_{l \in \Z}E_{i,l-i}.
\end{align*}
We set
\begin{align*}
\fo_J^{\odd}(k)=
&\fg J(k)^{\tau_0^s, -}=\{X \in \fg J(k) \vert \tau_0^s(X)=-X\},  \\
\fsp_J(k)=
&\fg J(k)^{\tau_{-1}, -}=\{X \in \fg J(k) \vert \tau_{-1}(X)=-X\}, \\
\fo_J^{\even}(k)=
&\fg J(k)^{\tau_{-1}^s, -}=\{X \in \fg J(k) \vert \tau_{-1}^s(X)=-X\}.
\end{align*}
\begin{remark}
The universal central extension of the Lie algebras $\fg J(k), \fo_J^{\odd}(k), \fsp_J(k)$ and $\fo_J^{\even}(k)$ are the Lie algebras of type $A_J, B_J, C_J$ and $D_J$ respectively that are used to obtain the Hirota bilinear forms with these symmetries. See, e.g., \cite{JM}, for details. 
\end{remark}
\subsection{Stable limit of Orthogonal and Symplectic Lie algebras and homology}
Set
\[ J_B=\sum_{i \in \Z}e_{i,-i}, \qquad J_C=\sum_{i \in \Z}(-1)^{i}e_{i,-i-1}, 
\qquad J_D=\sum_{i \in \Z}e_{i,-i-1}. \]
We define the anti-involutions $\tau_B, \tau_C$ and $\tau_D$ of $\fgl_\infty(R)$ as follows:
\[ \tau_B(X)=J_B({}^tX)J_B, \qquad \tau_C(X)=-J_C({}^tX)J_C, \qquad
\tau_D(X)=J_D({}^tX)J_D. \]
The Lie algebras
\[ \fo_{\odd}(R):=\fgl_\infty(R)^{\tau_B, -}, \qquad \fsp(R):=\fgl_\infty(R)^{\tau_C, -}, 
\qquad \fo_{\even}(R):=\fgl_\infty(R)^{\tau_D, -}
\]
are clearly the stable limit of the Lie algebras $\{\fo_{2l+1}(R)\}_{l \in \Z_{>0}}, \{\fsp_{2l}(R)\}, \{\fo_{2l}(R)\}_{l \in \Z_{>1}}$, respectively. In a way similar to the case $\fg J(k)$ (cf. Lemma \ref{lemma_stable1}), one can show 
\begin{lemma}\label{lemma_stable2} There exists isomorphisms of homologies:
\begin{enumerate}
\item $H_\bullet(\fo^{\odd}_J(k)) \cong H_\bullet(\fo_{\odd}(J(k)))$, 
\item $H_\bullet(\fsp_J(k)) \cong H_\bullet(\fsp(J(k)))$, 
\item $H_\bullet(\fo^{\even}_J(k)) \cong H_\bullet(\fo_{\even}(J(k)))$, 
\end{enumerate}
\end{lemma}
Thanks to Theorem 5.5 in \cite{LP}, due to J. L. Loday and C. Procesi, we obtain
\begin{theorem}\label{thm_LP} Let $\ast:J(k) \rightarrow J(k)$ be the anti-involution satisfying
\[ E_{r,s}(a)^\ast=E_{-s,-r}(a) \qquad a\in k. \]
Then, for $\fg=\fo_J^{\odd}, \fsp_J$ and $\fo_J^{\even}$, we have
\[ \prim(H_\bullet(\fg(k)))={}_{-1}HD_{\bullet-1}(J(k)), 
\]
where ${}_{-1}HD_\bullet(\cdot )$ signifies the skew-dihedral homology. 
\end{theorem}
It is slightly technical, but one can show (cf. \cite{FI2}) that
\[ {}_{-1}HD_{\bullet-1}(J(k))=HD_{\bullet-2}(k). \]
Notice that
$HD_p(k)=k$ for $p \equiv 0 \,\text{mod}\, 4$ and $HD_p(k)=0$ otherwise. Thus, we have
\begin{theorem}[cf. \cite{FI2}]\label{thm_main2} For $\fg=\fo_J^{\odd}, \fsp_J$ and $\fo_J^{\even}$, $\prim(H_\bullet(\fg(k)))$ is the graded $k$-vector space whose $p$th graded component is 
\[ \prim(H_\bullet(\fg(k)))_p=\begin{cases} \, k\, \quad & p \equiv 2 (4), \\ \, 0 \, \quad & \text{otherwise}. \end{cases}
\]
\end{theorem}
\begin{remark} By Theorems \ref{thm_main1} and \ref{thm_main2}, the support of 
$\prim(H_\bullet(\fg(k)))$ for $\fg=\, \fg J, \, \fo_J^{\odd}, $\, $\fsp_J$ and $\fo_J^{\even}$, i.e., those integers $p$ where $\prim(H_\bullet(\fg(k)))_p \neq 0$, are exactly the twice of the \emph{exponents} of the Weyl group of type $A_J, B_J, C_J$ and $D_J$, respectively. 
\end{remark}

\section{Generalizations and some further topics}\label{sect_5}
Here, we give a generalization on the coefficient ring, mention  some other related topics and open questions. 

\subsection{Lie algebras with coefficients in $R$}
The computations of Lie algebra homology given in Sections \ref{sect_2} and \ref{sect_3} can be generalized to the Lie algebra $\fg(R)$, where $\fg=\fg J, \fo_J^{\odd},\, \, \fsp_J$ and $\fo_J^{\even}$. Here $R$ is an associative unital $k$-algebra (equipped with a $k$-linear anti-involution $\bar{\cdot}: R \rightarrow R$ for $ \fg=\fo_J^{\odd}, \, \fsp_J$ and $\fo_J^{\even}$). 

The next theorem generalizes Theorem \ref{thm_main1}:
\begin{theorem}[cf. \cite{FI1}]\label{thm_main1-g} 
\begin{enumerate}
\item $\prim(H_\bullet(\fg J(R)))=HC_{\bullet-2}(R)$, 
\item $H_\bullet(\fg J(R)) \cong  S(HC_{\bullet-2}(R))$. 
\end{enumerate}
\end{theorem}
The proof goes as follows. As in the case $R=k$, By Theorem \ref{thm_LQT}, the primitive part $\prim(H_\bullet(\fg J(R)))$ is given by $HC_{\bullet-1}(J(R))$. Now, D. Stefan's spectral sequence allows us to show an isomorphism $HH_{\bullet}(J(R)) \cong HH_{\bullet-1}(R)$. 
Analyzing carefully this isomorphism, it can be shown that $HC_{\bullet-1}(J(R))\cong HC_{\bullet-2}(R)$. 

The above proof also allows us to obtain a generalization of Theorem \ref{thm_main2}:
\begin{theorem}[cf. \cite{FI2}]\label{thm_main2-g} Set $\fg=\fo_J^{\odd}, \fsp_J$ and $\fo_J^{\even}$.
\begin{enumerate}
 \item $\prim(H_\bullet(\fg(k)))=HD_{\bullet-2}(R)$, 
 \item $H_\bullet(\fg(R)) \cong S(HD_{\bullet-2}(R))$.
 \end{enumerate}
 \end{theorem}
 
\subsection{Rank and trace functionals}
For any $n \in \Z_{>1}$, let $\fg J_n(k)$ be the image of the composition $\widetilde{\fg J_n}(k) \hookrightarrow \widetilde{\fg J}(k) \twoheadrightarrow \fg J(k)$. Let $\fg J_\infty(k)$ be the subalgebra of $\fg J(k)$ generated by all of the $\fg J_n(k)$ ($n \in \Z_{>1}$).  Viewed as an associative algebra,  $\fg J_n(k)$ and $\fg J_\infty(k)$ will be denoted by $J_n(k)$ and $J_\infty(k)$, respectively. We don't know whether this algebra $J_\infty(k)$ is a \emph{von Neumann regular ring}, i.e., for any $A \in J_\infty(k)$, there exists $X \in J_\infty(k)$ such that $AXA=A$. Nevertheless, there is a so-called \emph{rank functional} defined as follows. For $A=(a_{i,j})_{i,j \in \Z} \in J_\infty(k)$, set
\[ \Rank(A)= \lim_{n \rightarrow \infty} \frac{1}{2n+1}\rank(A_n), \]
where, the matrix $A_n$ of size $2n+1$ is defined by $(a_{i,j})_{-n \leq i,j\leq n}$ and $\rank(\,\cdot\,)$ signifies the rank of a finite size matrix. 
\begin{remark}
It was pointed out by B. Feigin and B. Tsygan in \cite{FT} that one can define the \emph{trace functional} $\Tr: J_\infty(k) \rightarrow k$ by 
\[ \Tr(A)=\lim_{n \rightarrow \infty} \frac{1}{2n+1}\tr(A_n),
\]
where $\tr( \cdot)$ is the trace of a finite size matrix. They used this functional to describe 
non trivial cocycles of the cohomology $H^\bullet(\fg J_\infty(k))$.
\end{remark}

\smallskip

It turns out that $\mathrm{Im}(\Rank)=[0,1] \cap \Q$. Indeed, this rank functional is a rank function on $J_\infty(k)$ in the sense of J. von Neumann \cite{vN}, i.e., it satisfies,
\begin{enumerate}
\item $\Rank(I)=1$, 
\item $\Rank(xy)\leq \Rank(x), \Rank(y)$,
\item $\Rank(e+f)=\Rank(e)+\Rank(f)$ for all orthogonal idempotents $e,f \in J_\infty(k)$, and
\item $\Rank(x)=0$ if and only if $x=0$.
\end{enumerate}
Hence, let $\widehat{J_\infty}(k) \subset J(k)$ be the $\Rank$-completion of $J_\infty(k)$. 
We denote the continuous extension of $\Rank(\, \cdot\,)$ to $\widehat{J_\infty}(k)$ by $\overline{\Rank}(\,\cdot\,)$ (cf. \cite{G}). 
\begin{proposition} $\mathrm{Im}(\overline{\Rank})=[0,1]$. 
\end{proposition}
\begin{proof}
It suffices to show that for any $x \in [0,1] \cap \Q^c$, there exists a diagonal $D=\mathrm{diag}(d_i)_{i \in \Z}$ such that $\overline{\Rank}(D)=x$.  \\
Let $\{r_n\}_{n \in \Z_{\geq 0}} \subset \Z_{>0}$ be the sequence defined as follows. Set $r_0=1$. For $n>0$, let $r_n$ be the integer such that 
$\frac{r_n-1}{2n+1}<x<\frac{r_n}{2n+1}$. Such an integer exists since $x$ is irrational. It can be checked that if $0<x<\frac{1}{2}$, then $r_{n+1}-r_n \in \{0,1\}$, otherwise, $r_{n+1}-r_n \in \{1,2\}$. \\
Now, we define the sequence $\{d_i\}_{i \in \Z}$. Set $d_0=1$. Suppose that $\{d_i\}_{\vert i\vert \leq n}$ is defined. Choose $d_{\pm (n+1)} \in \{0,1\}$ in such a way that $\sharp\{ i \in \{\pm(n+1)\}\vert d_i=1\}=r_{n+1}-r_n$. 
For any such choice, it follows that the rank of the diagonal matrix $\mathrm{diag}(d_i)_{\vert i\vert \leq n}$ is $r_n$ for any $n \in \Z_{\geq 0}$. Hence, we have $\overline{\Rank}(D)=x$ by construction. 
\end{proof}

There might be a von Neumann regular subalgebra of $J(k)$ that has not been explored up to now.  
\smallskip

Here are some questions: 
\begin{enumerate}
\item Does  $\widehat{J_\infty}(k)$ admit a Lie algebra structure ? If yes,
\begin{enumerate}
\item[i)] and if the Lie algebra $\fg \widehat{J_\infty}(k)$ is perfect, describe its universal central extension.
\item[ii)] More generally, can we have a description of the (co)homology ring of Lie algebra $\fg \widehat{J_\infty}(k)$ ?
\end{enumerate}
\item Is $J_\infty(k)$ von Neumann regular algebra ? If yes, so is $\widehat{J_\infty}(k)$ (cf. \cite{G}).
\begin{enumerate}
\item[i)] What kind of additional properties, do they have ?
\end{enumerate}
\end{enumerate}

\subsection{More about $\fg J_\infty(k)$}
In \cite{FT}, B. Feigin and B. Tsygan determined the structure of the cohomology ring $H^\bullet(\fg J(k))$ which states, for each $i>0$, there exists, up to scalar,  unique $c_i \in H^{2i}(\fg J(k))$ such that 
\[ H^\bullet(\fg J (k)) \cong S^\bullet(\bigoplus_{i>0} kc_i). \]
This follows from Theorem \ref{thm_main1}.  The natural inclusion $\fg J_\infty(k) \hookrightarrow \fg J(k)$ induces a surjective map $H^\bullet(\fg J_\infty(k)) \twoheadrightarrow H^\bullet(\fg J (k))$. Indeed, the showed that, for each $i>0$, there exists, up to scalar,  unique $\xi_i \in H^{2i-1}(\fg J_\infty(k))$ such that the next short sequence is exact: 
\[ 0 \longrightarrow \bigwedge{}^\bullet(\bigoplus_{i>0}k \xi_i) \longrightarrow H^\bullet(\fg J_\infty(k))
\longrightarrow H^\bullet(\fg J(k)) \longrightarrow 0. \]
They even presented an explicit realization of cocycles. 

It may be an interesting problem to obtain such descriptions also for the Lie algebras $\fg_{J_\infty}(k):=\fg_J(k) \cap \fg J_\infty(k)$, where $\fg_J=\fo_J^{\odd}, \fsp_J$ and $\fo_J^{\even}$. 

\subsection*{Acknowledgments}
The authors thank Max Karoubi and Andrey Lazarev for helpful discussions.  Kenji Iohara would like to thank the organizers for giving him the opportunity to present some of the results in this article.

\end{document}